\documentclass{amsart}
\usepackage{amssymb}
\usepackage{pb-diagram}

\newtheorem{theorem}{Theorem}[section]
\newtheorem{thm}[theorem]{Theorem}
\newtheorem{prop}[theorem]{Proposition}
\newtheorem{lem}[theorem]{Lemma}

\newtheorem{fact}[theorem]{Fact}
\newtheorem{cor}[theorem]{Corollary}

\theoremstyle{definition}
\newtheorem{defn}[theorem]{Definition}

\newtheorem*{thmA}{Theorem A}

\newtheorem*{thmB}{Theorem B}
\newtheorem*{thmC}{Theorem C}
\newtheorem*{thmD}{Theorem D}
\newtheorem*{thmE}{Theorem E}

\theoremstyle{remark}

\newcommand{\dcl}{\operatorname{dcl}}

\newcommand{\comp}{\operatorname{compl}}

\newcommand{\cal}[1]{\ensuremath{\mathcal{#1}}}
\newcommand{\Cal}[1]{\ensuremath{\mathcal{#1}}}

\newcommand{\N}{\mathbb{N}}
\newcommand{\Q}{\mathbb{Q}}
\newcommand{\R}{\mathbb{R}}

\newcommand{\Int}{\operatorname{Int}}
\newcommand{\cl}{\operatorname{Cl}}
\newcommand{\gr}{\operatorname{gr}}
\newcommand{\tp}{\operatorname{tp}}

\title{Wild theories with o-minimal open core}

\begin{document}
\date{\today}

\author{Philipp Hieronymi}
\address
{Department of Mathematics\\University of Illinois at Urbana-Champaign\\1409 West Green Street\\Urbana, IL 61801}
\email{phierony@illinois.edu}
\urladdr{http://www.math.illinois.edu/\textasciitilde phierony}

\author{Travis Nell}
\email{tnell2@illinois.edu}
\urladdr{http://www.math.illinois.edu/\textasciitilde tnell2}

\author{Erik Walsberg}
\email{erikw@illinois.edu}
\urladdr{http://www.math.illinois.edu/~erikw}
\date{\today}


\thanks{The first author was partially supported by NSF grant DMS-1654725.}
 \maketitle

\begin{abstract}
Let $T$ be a consistent o-minimal theory extending the theory of densely ordered groups and let $T'$ be a consistent theory. Then there is a complete theory $T^*$ extending $T$ such that $T$ is an open core of $T^*$, but every model of $T^*$ interprets a model of $T'$. If $T'$ is NIP, $T^*$ can be chosen to be NIP as well. From this we deduce the existence of an NIP expansion of the real field that has no distal expansion.
\end{abstract}

\section{Introduction}
Let $\Cal R$ be an expansion of a dense linear order $(R,<)$ without endpoints. The \textbf{open core of $\Cal R$}, denoted by $\Cal R^{\circ}$, is the structure $(R,(U))$, where $U$ ranges over all open sets of all arities definable in $\Cal R$. Miller and Speissegger introduced this notion of an open core for expansions of $(\R,<)$  in \cite{MS99}, and established sufficient conditions on $\Cal R$ such that its open core is o-minimal.  Here we want to answer the following question:
\begin{center}
\emph{Is there any restriction on what kind of structures can be interpreted in an expansion of $(R,<)$ with o-minimal open core?}
\end{center}
This question, although formulated slightly differently, was already asked by Dolich, Miller and Steinhorn in a preprint version of \cite{DMS-Indepedent}. Our answer is negative. To give a precise statement of our result, we need to recall the notion of an open core of a theory as introduced in \cite{DMS1}.
Let $T^*$ be a theory extending the theory of dense linear orders without endpoints in a language $\Cal L^*\supseteq \{ < \}$, and let $T$ be another theory in a language $\Cal L$. We say that \textbf{$T$ is an open core of $T^*$} if for every $\Cal N\models T^*$ there is $\Cal M \models T$ such that $\Cal N^{\circ}$ is interdefinable with $\Cal M$.

\begin{thmA} Let $T$ be a consistent o-minimal theory extending the theory of densely ordered
groups and let $T'$ be a consistent theory. Then there is a complete theory $T^*$ extending $T$ such that
\begin{enumerate}
\item $T^*$ interprets a model of $T'$,
\item $T$ is an open core of $T^*$,
\item $T^*$ is NIP if $T'$ is NIP,
\item $T^*$ is strongly dependent if $T'$ is strongly dependent.
\end{enumerate}
\end{thmA}

\noindent Statements (3) and (4) of Theorem A indicate that we can choose $T^*$ in such a way that not only the open core of $T^*$ is o-minimal, but also $T^*$ remains tame in the sense of Shelah's combinatorial tameness notions. For definitions of NIP and strong dependence, we refer the reader to Simon \cite{Simon-Book}. \newline

\noindent We will deduce the following analogue for o-minimal expansions of the ordered real additive group from the proof of Theorem A.

\begin{thmB} Let $\Cal R$ be an o-minimal expansion of $(\R,<,+)$ in a language $\Cal L$ and let $T'$ be a consistent theory such that $|\Cal L|<|\R|$ and $|T'|\leq |\R|$. Then there exists an expansion $\Cal S$ of $\Cal R$ such that
\begin{enumerate}
\item $\Cal S$ interprets a model of $T'$,
\item the open core of $\Cal S$ is interdefinable with $\Cal R$,
\item $\Cal S$ is NIP if $T'$ is NIP,
\item $\Cal S$ is strongly dependent if $T'$ is strongly dependent.
\end{enumerate}
\end{thmB}

\noindent We will deduce from work in \cite{DMS1} that an expansion of $(\R,<,+)$ has o-minimal open core if and only if it does not define a discrete linear order. Therefore in Theorem B the statement \emph{$\Cal S$ interprets a model of $T'$} cannot be replaced by the statement \emph{$\Cal S$ defines a model of $T'$}. \newline

\noindent The outline of the proof of the above results is as follows. For simplicity, let $\Cal R$ be $(\R,<,+)$ and let $T'$ be a consistent theory in a countable language $\Cal L'$ with an infinite model. Take a dense basis $P$ of $\R$ as a $\Q$-vector space. By \cite[2.25]{DMS-Indepedent} the open core of the structure $(\R,<,+,P)$ is $\Cal R$. We further expand $(\R,<,+,P)$ by a binary predicate $E$ such that $E$ is an equivalence relation on $P$, has countably many equivalence classes and each equivalence class of $E$ is dense in $P$. Now take a countable model $M$ of $T'$ and expand $(\R,<,+,P,E)$ to an expansion $\Cal S$ such that the quotient $P/E$ becomes an $\Cal L'$-structure that is isomorphic to $M$. Since each equivalence class of $E$ is dense in $P$ and hence in $\R$, we can define this \emph{fusion} $\Cal S$ of $(\R,<,+,P,E)$ and $M$ in a way that the open core of the resulting structure $\Cal S$ is still $\Cal R$. Indeed we use ideas and techniques from \cite{DMS-Indepedent} to prove a quantifier-elimination result for $\Cal S$ analogous to the one of $(\R,<,+,P)$ (see \cite[2.9]{DMS-Indepedent}), and from that deduce that the open core of $\Cal S$ is $\Cal R$.
\newline

\noindent In the special case that $\Cal L'$ is empty and $T'$ is the theory of infinite sets, the construction we outlined above gives the following extension of the results from \cite{DMS-Indepedent}.

\begin{thmC}
Let $T$ be a complete o-minimal theory extending the theory of densely ordered
groups in a language $\Cal L$, and let $\Cal L_e$ be the language $\Cal L$ augmented by a unary predicate $P$ and a binary predicate $E$.
Let $T_{e,\infty}$ be the $\Cal L_e$-theory containing $T$ and axiom schemata expressing the following statements:
\begin{itemize}
\item[(1)] $P$ is dense and $\dcl_T$-independent,
\item[(2)] $E \subseteq P^2$ is an equivalence relation on $P$,
\item[(3)] each equivalence class of $E$ is dense in $P$,
\item[(4)] $E$ has infinitely many equivalence classes.
\end{itemize}
Then $T_{e,\infty}$ is complete, and $T$ is an open core of $T_{e,\infty}$.
\end{thmC}

\noindent Theorem B should be compared to Friedman and Miller \cite[Theorem A]{FM-Sparse}. Among other things, the latter result implies the existence of an expansion of the real field that defines a model of first-order arithmetic, but every subset of $\R$ definable in this expansion is a finite union of an open set and finitely many discrete sets. Therefore both our result and \cite{FM-Sparse} describe situations in which topological tameness exists without model-theoretic tameness.\newline

\noindent In general our results rule out that the property of having an o-minimal open core has any consequences in terms of model-theoretic tameness of the whole structure. At first glance this might look like a disappointing result. However, we do not share this viewpoint. We regard our results as further evidence that in model-theoretically wild situations geometric tameness can often prevail. In some of those situations the open core of a structure or theory seems to be the right tool that can capture precisely this tameness, making certain phenomena trackable by model-theoretic analysis.\newline

\noindent Theorem B(3) has a few interesting corollaries about NIP expansions of $(\R,<,+)$. First of all, it states that for every NIP theory $T'$ of cardinality at most continuum there is an NIP expansion of $(\R,<,+)$ that interprets a model of $T'$. Therefore the model theory of NIP expansions of $(\R,<,+)$ is in general as complicated as the model theory of arbitrary NIP theories.  We use this observation to deduce a new result about the distality of NIP expansions of $(\R,<,+)$. The notion of distality was introduced by Simon in \cite{Simon-Distal} to single out those NIP theories and structures that can be considered purely unstable. While every o-minimal expansions of $(\R,<,+)$ is distal, there are several natural examples of non-distal NIP expansions of $(\R,<,+)$ (see \cite{HN-Distal}). However, by Chernikov and Starchenko \cite{CS-RegDis} even just having a distal expansion guarantees certain desirable combinatorial properties of definable sets (the strong Erd\"os-Hajnal property). Therefore it is interesting to know whether or not all NIP expansions of $(\R,<,+)$ have a distal expansion. Although we do not know it, we expect all examples of non-distal NIP expansions of $(\R,<,+)$ produced in \cite{HN-Distal} to have distal expansions. So far the only known NIP theory without an distal expansion is the theory of algebraically closed fields of characteristic $p$ by \cite[Proposition 6.2]{CS-RegDis}. Combining this with Theorem B, we almost immediately obtain the following.

\begin{thmD} There is an NIP expansion of $(\R,<,+)$ that does not have a distal expansion.
\end{thmD}

\noindent This is also the first example of an NIP expansion of any densely ordered set that does not have a distal expansion. \newline

\noindent While in general for every countable NIP theory there is an expansion of $(\R,<,+)$ that interprets a model of this theory, there is a natural class of expansions of $(\R,<,+)$ in which models of certain NIP theories in countable languages cannot be interpreted. A set $X\subseteq \R$ is somewhere dense and co-dense if there is an open interval $I$ such that $X\cap I$ is dense and co-dense in $I$. We say an expansion of $(\R,<)$ is \textbf{noiseless} if it does not define a somewhere dense and co-dense subset of $\R$.\footnote{The name \emph{noiseless} was suggested by Chris Miller. Being noiseless is equivalent to the statement that every definable subset of $\R$ either has interior or is nowhere dense. The latter condition has also been called \emph{i-minimality} by Fornasiero \cite{F-tameopen}.} The expansion $\Cal S$ we produce for Theorem B is not noiseless. It is therefore natural to ask whether in Theorem B we can require $\Cal S$ to be noiseless. The answer to this question is negative.

\begin{thmE} Let $\Cal R$ be a noiseless NIP expansion of $(\R,<,+,1)$. Then $\Cal R$ has definable choice, that is: for $A\subseteq \R^{m}\times \R^n$ $\emptyset$-definable in $\Cal R$ there is an $\emptyset$-definable function $f: \pi(A) \to \R^n$ such that
\begin{enumerate}
\item $\gr(f)\subseteq A$,
\item $f(a) = f(b)$ whenever $a,b\in \pi(A)$ and $A_a=A_b$,
\end{enumerate}
where $\pi : \R^{m+n} \to \R^m$ is the projection onto the first $m$ coordinates.
\end{thmE}

\noindent It follows from Theorem E that if a noiseless NIP expansion $\Cal R$ of $(\R,<,+,1)$ interprets a structure $\Cal M$, then $\Cal R$ defines an isomorphic copy of $\Cal M$. We will prove Theorem E in greater generality. In particular, Theorem E not only holds for noiseless NIP expansions, but also for noiseless NTP$_2$ expansions (for a definition of NTP$_2$ see \cite{Simon-Book}).\newline

\noindent We now show that Theorem B fails when we require $\Cal S$ to be noiseless. Let $p$ be a prime and $\mathbb{F}_p$ be the field with $p$ elements.
By Shelah and Simon \cite[Theorem 2.1]{shelahsimon} if $\Cal V = (V, +, \ldots)$ is an infinite $\mathbb{F}_p$-vector space and $\prec$ is a linear order on $V$, then $(\Cal V, \prec)$ has IP.
Suppose now that $\Cal M = (M, <, \ldots)$ is an expansion of an infinite linear order $(M,<)$ and that $\cal V$ is an $\Cal M$-definable infinite $\mathbb{F}_{p}$-vector space with underlying set $V \subseteq M^k$.
The lexicographic order on $M^k$ induced by $<$ is linear and induces a linear order on $V$.
It follows that $\Cal M$ has IP.
Thus no NIP expansion of a linear order defines an infinite $\mathbb{F}_p$-vector space.
By Theorem E no noiseless NIP expansion of $(\R,<,+,1)$ interprets an infinite vector space over a finite field.

\subsection*{Open questions} We end the introduction with a few open questions.\newline

\noindent \textbf{1.} We work here in the context of ordered structures and o-minimal open core. It is likely that our techniques can be used to extend our results to various other settings. In particular, by using the technology from Berenstein and Vassiliev \cite{BV-Selecta} rather than from \cite{DMS-Indepedent} one should be able produce analogues of Theorem A and B for other geometric structures such as the field of $p$-adic numbers.\newline

\noindent \textbf{2.} Similar questions can be asked about NIP expansions of $(\N,<)$. Since every such expansion has definable Skolem functions, we again have some limitations on what kind of theories can be interpreted in such a structure. Can we say anything more? For example: can an NIP expansion of $(\N,<)$ interpret an infinite field? Is there an NIP expansion of $(\N,<)$ that does not admit a distal expansion?\newline

\noindent \textbf{3.} Is there a noiseless NIP expansion of $(\R,<,+)$ that does not admit a distal expansion? Is every infinite field interpretable in a noiseless NIP expansion isomorphic to $(\R,+,\cdot)$ or $(\mathbb{C}, +, \cdot)$?\newline

\noindent The previous question is even open for d-minimal NIP expansions, a subclass of the class of noiseless NIP expansions (see \cite{Miller-tame} for a definition of d-minimality). It follows from Fornasiero \cite[Theorem 4.13]{F-dmingroup} that any uncountable field interpretable in a d-minimal expansion is isomorphic to $(\R,+,\cdot)$ or $(\mathbb{C},+,\cdot)$.
Thus in this setting it suffices to show that no d-minimal NIP expansion interprets a countable field.
It is not difficult to show that any countable set definable in a d-minimal expansion admits a definable order with order type $\omega$.
Thus, if the above question about the interpretability of infinite fields in NIP expansions of $(\N,<)$ has a negative answer, then any infinite field interpretable in a d-minimal NIP expansion of $(\R,<,+)$ is isomorphic to $(\R,+,\cdot)$ or $(\mathbb{C},+,\cdot)$.\newline

\noindent Is every noiseless NIP expansion of $(\R,<,+)$ d-minimal? We doubt that this statement is true, but it seems difficult to produce a counterexample.

\subsection*{Acknowledgments} The authors thank Antongiulio Fornasiero and Chris Miller for helpful conversations around the topic of this paper.

\subsection*{Notation}  We will use $m,n$ for natural numbers and $\kappa$ for a cardinal. Let $X,Y$ be sets. We denote the cardinality of $X$ by $|X|$. For a function $f: X\to Y$, we denote the graph of $f$ by $\gr(f)$. If $Z\subseteq X \times Y$ and $x\in X$, then $Z_x$ denotes the set $\{ y\in Y \ : \ (x,y) \in Z\}$. If $a=(a_1,\dots,a_n)$, we sometimes write $Xa$ for $X\cup \{a_1,\dots,a_n\}$, and $XY$ for $X\cup Y$.\newline
Let $\Cal L$ be a language and $T$ an $\Cal L$-theory. Let $M \models T$ and $A\subseteq M$. In this situation, $\Cal L$-definable always means $\Cal L$-definable with parameters. If we want to be precise about the parameters we write $\Cal L$-$A$-definable to indicate $\Cal L$-definability with parameters from $A$. Let $b\in M^n$. Then we write $\tp_{\Cal L}(b|A)$ for the $\Cal L$-type of $b$ over $A$. Moreover, $\dcl_T(A)$ denotes the definable closure of $A$ in $M$. Whenever $T$ is o-minimal, $\dcl_T$ is a pregeometry.

\section{The fusion}\label{Section:Fusion}

Let $T$ be a consistent o-minimal theory extending the theory of densely ordered
groups with a distinguished positive element, and let $\Cal L$ be its language.  Let $\Cal L'$ be a relational language disjoint from $\Cal L$, and let $T'$ be a consistent $\Cal L'$-theory. In this section we will construct a language $\Cal L^* \supseteq \Cal L$ and a complete $\Cal L^*$-theory $T^*$ extending $T$ such that $T$ is an open core of $T^*$ and $T^*$ interprets $T'$.  In Section \ref{Section:NIP} we show that $T^*$ is NIP whenever $T$ is, and in Section \ref{Section:STDEP} we prove that strong dependence of $T'$ implies strong dependence of $T^*$.\newline

\noindent By replacing $T$ by a completion of $T$ and $T'$ by a completion of $T'$, we can directly reduce to the case that both $T$ and $T'$ are complete. So from now, we assume that $T$ and $T'$ are complete.\newline

\noindent Let $\Cal L_e$ be $\Cal L$ expanded by a unary predicate $P$ and a binary predicate $E$ such that neither $P$ nor $E$ are not in $\Cal L'$. Let $T_{e}$ be the extension of $T$ by axiom schemata expressing the following statements:
\begin{itemize}
\item[(T1)] $P$ is dense and $\dcl_T$-independent,
\item[(T2)] $E \subseteq P^2$ is an equivalence relation on $P$,
\item[(T3)] each equivalence class of $E$ is dense in $P$.
\end{itemize}

\noindent Let $\Cal L^* = \Cal L_e \cup \Cal L'$. For a given $\Cal L'$-formula $\theta$ we define a $\Cal L^*$-formula $\theta_e$ recursively as follow:
\begin{align*}
&\hbox{if $\theta$ is $x=y$, then define $\theta_e$ as $Exy$,}\\
&\hbox{if $\theta$ is $Rx_1\dots x_n$ where $R$ is an $n$-ary predicate in $\Cal L'$, then define $\theta_e$ as $Rx_1\dots x_n$,}\\
&\hbox{if $\theta$ is $\neg\theta'$, then define $\theta_e$ as $\neg \theta'_e$,} \\
&\hbox{if $\theta$ is $\theta'\wedge \theta''$, then define $\theta_e$ as $\theta'_e\wedge \theta''_e$,}   \\ \displaybreak
&\hbox{if $\theta$ is $\theta'\vee \theta''$, then define $\theta_e$ as $\theta'_e\vee \theta''_e$,} \\
&\hbox{if $\theta$ is $\exists x \theta'$, then define $\theta_e$ as $\exists x (Px \wedge \theta'_e),$}\\
&\hbox{if $\theta$ is $\forall x \theta'$, then define $\theta_e$ as $\forall x (Px \rightarrow \theta'_e).$}
\end{align*}

\noindent Let $T^*$ be the extension of $T_e$ by the following axiom schemata:
\begin{itemize}
\item[(T4)] $R\subseteq P^n$ and
\[
\forall x_1 \forall y_1\dots \forall x_n \forall y_n \ \left( \bigwedge_{i=1}^n Ex_iy_i \right) \rightarrow \Big(Rx_1\dots x_n \leftrightarrow Ry_1\dots y_n\Big)
\]
for every $R \in \Cal L'$ with $\operatorname{ar}(R)=n$,
\item [(T5)] $\varphi_e$ for every $\varphi \in T'$.
\end{itemize}

\vspace{0.2cm}

\noindent We now fix some further notation. Given a model $\Cal M$ of  $T^*$, we will denote the underlying model of $T$ by $M$, the interpretation of $P$ and $E$ by $P_{\Cal M}$ and $E_{\mathcal M}$. \newline
For $b\in P_{\Cal M}^n$ and $A\subseteq P_{\Cal M}$ we denote by $\tp_{\Cal L'}(b|A)$ the set of all $\Cal L^*$-formulas of the form $\varphi_e(x,a)$ for some $\Cal L'$-formula $\varphi(x,y)$ such that $a \in A^m$ and $\Cal M \models \varphi_e(b,a)$.\newline

\noindent A standard induction on $\Cal L'$-formulas together with Axiom (T4) gives the following.

\begin{lem}\label{lem:sametype} Let $\Cal M\models T^*$, $a,b \in P_{\Cal M}$ and $A\subseteq P_{\Cal M}$ . If $(a,b) \in E_{\Cal M}$, then $\tp_{\Cal L'}(a|A)=\tp_{\Cal L'}(b|A)$.
\end{lem}

\noindent We now show that given a model of $T$ with enough $\dcl_T$-independent elements, this model can be expanded to a model of $T^*$. This result will be used to show consistency of $T^*$.

\begin{lem}\label{lem:expand} Let $M\models T$ and let $(A_b)_{b\in B}$ be a family of dense subsets of $M$ such that
\begin{itemize}
\item $\bigcup_{b\in B} A_b$ is $\dcl_T$-independent,
\item $A_b \cap A_{b'} = \emptyset$ whenever $b\neq b'$,
\item there is a model of $T'$ with the same cardinality as $B$.
\end{itemize}
Then $M$ can be expanded to a model of $T^*$.
\end{lem}
\begin{proof}
Let $\Cal N$ be a model of $T'$ with the same cardinality as $B$. Without loss of generality, we can assume that $B$ is the universe of $\Cal N$.
We now expand $M$ to an $\Cal L^*$-structure $\Cal M$. We interpret the relation symbol $P$ as $P_{\mathcal M}:= \bigcup_{b\in B} A_b$. For $a,a'\in P_{\Cal M}$ we say $a E_{\Cal M} a'$ if and only if there is $b\in B$ such that $a,a'\in A_b$.
It is clear that $E_{\Cal M}$ is an equivalence relation on $P_{\Cal M}$ and that every equivalence class of $E_{\Cal M}$ is dense in $M$. Thus $(M,P_{\mathcal M},E_{\Cal M})$ is an $\Cal L_e$-structure that models $T_e$. It is left to interpret the elements of $\Cal L'$. Let $R$ be an $n$-ary relation symbol in $\Cal L'$. We define its interpretation $R_{\Cal M}$ by
\[
\{ (a_1,\dots,a_n) \in P_{\Cal M}^n  \ : \ \exists b_1,\dots,b_n \in B \  \bigwedge_{i=1}^n a_i \in A_{b_i} \wedge \Cal N \models R(b_1,\dots b_n)\}.
\]
Let $\Cal M := (M,P_{\Cal M},E_{\Cal M}, \big(R_{\Cal M}\big)_{R\in \Cal L'})$. It is clear from the definition of $E_{\Cal M}$ and  $R_{\Cal M}$ that $\Cal M$ satisfies (T4). By a straightforward induction on formulas we  see that for every $\Cal L'$-formula $\varphi(x)$ and for every $a_1,\dots,a_n \in P_{\Cal M}$ and $b_1,\dots,b_n\in B$ with $a_i \in A_{b_i}$ we have
\[
\Cal M \models \varphi_e(a_1,\dots,a_n) \hbox{ if and only if } \Cal N \models \varphi(b_1,\dots,b_n).
\]
Thus $\Cal M$ satisfies (T5), and therefore $\Cal M\models T^*$.

\end{proof}

\begin{prop}\label{prop:consistent} The theory $T^*$ is consistent.
\end{prop}
\begin{proof}
 By \cite[1.11]{DMS-Indepedent} there is a model $M$ of $T$ and a family $(A_b)_{b\in B}$ of dense subsets of $M$ such that the family $(A_b)_{b \in B}$ satisfies the assumptions of Lemma \ref{lem:expand}. The statement of the proposition then follows from Lemma \ref{lem:expand}.
\end{proof}

\begin{prop}\label{prop:interpret} Every model of $T^*$ interprets a model of $T'$.
\end{prop}
\begin{proof}
Let $\Cal M :=(M,P_{\Cal M},E_{\Cal M}, \big(R_{\Cal M}\big)_{R\in \Cal L'})\models T^*$. Let $N$ be the set of equivalence classes of $E_{\Cal M}$. For an $r$-ary relation symbol $R$, let
\[
R_{\Cal N} := \{ ([a_1]_{E_{\cal M}},\dots,[a_n]_{E_{\cal M}}) \in N^n \ : \ (a_1,\dots,a_n) \in R_{\Cal M}\}.
\]
Note that $R_{\Cal N}$ is well-defined by Lemma \ref{lem:sametype}. Let $\Cal N = (N, \big(R_{\Cal N}\big)_{R\in \Cal L'})$. Since $\Cal N$ is interpretable in $\Cal M$, it is only left to show that $\Cal N \models T'$. Using a straightforward induction on $\Cal L'$-formulas and Axiom (T4) the reader can check that for every $\Cal L'$-formula $\varphi(x)$ and for every $a_1,\dots,a_n \in P_{\Cal M}$
\[
\Cal M \models \varphi_e(a_1,\dots,a_n) \hbox{ if and only if } \Cal N \models \varphi([a_1]_{E_{\cal M}},\dots,[a_n]_{E_{\cal M}}).
\]
Thus $\Cal N \models T'$, since $\Cal M$ satisfies (T5).
\end{proof}

\noindent Proposition \ref{prop:interpret} shows that $T^*$ satisfies condition (1) of Theorem A. In the rest of this section we will show that $T^*$ also satisfies condition (2). In order to do so we have to carefully analyse the definable sets in models of $T^*$.

\subsection{Back-and-forth system} To better understand definable sets and types in models of $T^*$, we follow the general strategy of the proofs of \cite[2.8]{DMS-Indepedent} and van den Dries \cite[Theorem 2.5]{densepairs} by constructing a back-and-forth system between models of $T^*$. Let $\kappa$ be a cardinal larger than $|T^*|$. Let $\Cal M_1$ and $\Cal M_2$ be two $\kappa$-saturated models of $T^*$. Let $\Cal I$ be the set of all partial $\Cal L$-isomorphisms $\iota : X \to Y$ between $\Cal M_1$ and $\Cal M_2$ such that there are
\begin{itemize}
\item finite $A\subseteq P_{\Cal M_1}$ and $A'\subseteq P_{\Cal M_2}$,
\item finite $Z\subseteq \Cal M_1$ and $Z'\subseteq \Cal M_2$
\end{itemize}
with
\begin{itemize}
\item[(i)] $\iota(A)=A'$ and $\iota(Z)=Z'$.
\item[(ii)] $Z$ and $Z'$ are $\dcl$-independent over $P_{\Cal M_1}$ and $P_{\Cal M_2}$ respectively,
\item[(iii)] $X=\dcl_T(AZ)$ and $Y=\dcl_T(A'Z')$,
\item[(iv)]  for all $a_1,\dots, a_n \in A$,
\[
\Cal M_1 \models \varphi_e(a_1,\dots,a_n) \hbox{ if and only if } \Cal M_2 \models \varphi_e(\iota(a_1),\dots,\iota(a_n)).
\]
\end{itemize}

\noindent In the following we will show that $\Cal I$ is back-and-forth system of partial $\Cal L^*$-iso-\linebreak morphisms.

\begin{lem} Let $\iota : X \to Y \in \Cal I$ and let $A,Z \subseteq \Cal M_1$ and $A',Z'\subseteq \Cal M_2$ be such that $A,A,Z,Z'$ satisfy conditions (i)-(iv) above. Then $\iota$ is a partial $\Cal L^*$-isomorphism and $X \cap P_{\Cal M_1} = A$.
\end{lem}
\begin{proof}
We first show that $X \cap P_{\Cal M_1} = A$. Suppose there is $z \in (X\cap P_{\Cal M_1})\setminus A$. Since $P_{\Cal M_1}$ is $\dcl_T$-independent and $A\subseteq P_{\Cal M_1}$, we have that $z \notin \dcl_T(A)$. Thus $z \in \dcl_T(AZ)\setminus \dcl_T(A)$. Since $\dcl_T$ is a pregeometry and $z\in P_{\Cal M_1}$, this contradicts the $\dcl_T$-independence of $Z$ over $P_{\Cal M_1}$. Similarly we can show that $Y\cap P_{\Cal M_2}=A'$. Since $\iota(A)=A'$, it follows that $\iota(X\cap P_{\Cal M_1})=Y\cap P_{\Cal M_2}$. Since $(x=y)_e$ is $Exy$, we can easily deduce from (iv) that $\iota$ is an $\Cal L_e$-isomorphism. Applying (iv) once more, we see that $\iota$ is also an $\Cal L^*$-isomorphism.
\end{proof}

\begin{lem} The set $\Cal I$ is a back-and-forth system.
\end{lem}
\begin{proof}
Let $\iota : X \to Y \in \Cal I$ and $b \in \Cal M_1$. By symmetry it is enough to show that if $b\notin X$, then we can find $\iota'\in \Cal I$ extending $\iota$ such that $b$ is in the domain of $\iota'$. From now on, assume that $b\notin X$.
\newline

\noindent \textbf{Case I:} $b \in P_{\Cal M_1}$. Let $p$ be the collection of all $\Cal L'$-formulas $\varphi(x,\iota(a))$ such that $a\in A^n$ and $\Cal M_1 \models \varphi_e(b,a)$. By saturation of $\Cal M_2$ and since $\iota \in \Cal I$, there is $b' \in P_{\Cal M_2}$ such that $\Cal M_2 \models \varphi_e(b',\iota(a))$ for every $\varphi(x,\iota(a))\in p$. By density of the equivalence classes of $E$, we can take an element $b'' \in P_{\Cal M_2}$ such that $(b',b'') \in E_{\Cal M_2}$ and the cuts realized by $b$ in $X$ and by $b''$ in $Y$ correspond via $\iota$.
Thus $\iota$ extends to an $\Cal L$-isomorphism $\iota' : \dcl_T(ZAb) \to \dcl_T(ZA'b'')$ with $\iota'(b)=b''$. Since $(b',b'') \in E_{\Cal M_2}$, we have by Lemma \ref{lem:sametype} that $\Cal M_2 \models \varphi_e(b'',\iota(a))$ for every $\varphi(x,\iota(a))\in p$. It is clear from our choice of $b''$ that $\iota' \in \Cal I$.\newline

\noindent \textbf{Case II:} $b \in \dcl_T(ZAP_{\Cal M_1})$.
\noindent Let $a_1,\dots,a_m \in P_{\Cal M_1}$ be such that $b\in \dcl_T(ZAa_1\dots a_m)$. By applying Case I $m$ times, we can find an element $\iota' \in \Cal I$ extending $\iota$ such that $a_1,\dots,a_m$ are in the domain of $\iota'$. Since the domain of $\iota'$ contains $\dcl_T(ZAa_1\dots a_m)$, it also contains $b$.\newline

\noindent \textbf{Case III:} $b \notin \dcl_T(AP_{\Cal M_1})$.
By \cite[2.1]{DMS-Indepedent} and saturation of $\Cal M_2$, there exists an element $b'\in \Cal M_2\setminus \dcl_T(A'P_{\Cal M_2})$ such that the cuts realized by $b$ over $X$ and $b'$ over $Y$ correspond via $\iota$. Therefore we can find an $\Cal L$-isomorphism $\iota' : \dcl_T(AZb) \to \dcl_T(A'Z'b')$ extending $\iota$ and mapping $b$ to $b'$. It is easy to check that $\iota' \in \Cal I$.
\end{proof}

\subsection{Completeness and quantifier-reduction} We now use the back-and-forth system $\Cal I$ to deduce certain desirable properties of $T^*$. In particular, we show completeness of $T^*$ and a quantifier-reduction result.

\begin{thm} The theory $T^*$ is complete.
\end{thm}
\begin{proof}
By Proposition \ref{prop:consistent}, $T^*$ is consistent. In the previous section we constructed a back-and-forth system between any two $\kappa$-saturated models of $T^*$. This implies that two such models are elementary equivalent. Completeness of $T^*$ follows.
\end{proof}

\begin{defn} We call an $\Cal L^*$-formula $\chi(y)$ \textbf{special} if it is of the form
\[
\exists x \ Px \wedge \psi_e(x) \wedge \varphi(x,y),
\]
where $\psi$ is an $\Cal L'$-formula and $\varphi$ is an $\Cal L$-formula.
\end{defn}

\noindent We now establish that $T^*$ has quantifier-elimination up to boolean combinations of special formulas (compare this result and its proof to \cite[2.9]{DMS-Indepedent}) and \cite[Theorem 1]{densepairs}).

\begin{thm} Each $\Cal L^*$-formula is $T^*$-equivalent to a boolean combination of special formulas.
\end{thm}
\begin{proof}
Let $\Cal M$ be a $\kappa$-saturated model of $T^*$. Let $\Cal M_1 := \Cal M_2 :=\Cal M$ and let $\Cal I$ be the back-and-forth system between $\Cal M_1$ and $\Cal M_2$ constructed in the previous section.
Let $a=(a_1,\dots,a_n),b=(b_1,\dots,b_n) \in M^n$ be such that $a$ and $b$ satisfy the same special formulas. To establish the theorem it suffices to show that $\tp_{\Cal L^*}(a)=\tp_{\Cal L^*}(b)$. To prove the latter statement, it is enough to find $\iota\in I$ that maps $a$ to $b$. By permuting the coordinates we can assume there is $r \in \{0,\dots,n\}$ such that $a_1,\dots,a_r$ are $\dcl_T$-independent over $P_{\Cal M}$ and $a_{r+1},\dots,a_n \in \dcl_T(a_1\dots a_rP_{\Cal M})$. Since $a$ and $b$ satisfy the same special formulas, the reader can easily verify that $b_1,\dots,b_r$ are $\dcl_T$-independent over $P_{\Cal M}$. Let $m\in \N$ and $g=(g_1,\dots,g_m) \in P_{\Cal M}^m$ be such that $a_{r+1},\dots,a_n \in \dcl_T(a_1\dots a_rg)$.
For $i=r+1,\dots,n$, let $f_i : M^{r+m} \to M$ be an $\Cal L$-$\emptyset$-definable function such that $f_i(a_1,\dots,a_r,g) = a_i$. We will now find $h=(h_1,\dots,h_m) \in P_{\Cal M}^m$ such that
\begin{itemize}
\item[(1)] $\tp_{\Cal L'}(h)=\tp_{\Cal L'}(g)$,
\item[(2)] $f_i(b_1,\dots,b_r,h) = b_i$ for each $i=r+1,\dots,n$.
\end{itemize}
If we have such $h$, we can find an $\Cal L$-isomorphism $\iota : \dcl(a_1\dots a_rg) \to \dcl(b_1\dots b_rh)$ such that $\iota(g)=h$ and $\iota(a_i) = b_i$ for each $i=1,\dots,r$. Since $h$ satisfies (1) and each of the sets $\{a_1\dots a_r\}$ and $\{b_1\dots b_r\}$ is $\dcl_T$-independent over $P_{\Cal M}$, it is easy to check that $\iota \in \Cal I$. Because $h$ also satisfies (2), we get that $\iota(a_i)=b_i$ for $i=r+1,\dots,n$. Thus $\iota$ is the desired element of $\Cal I$.\newline

\noindent We now prove the existence of an $h \in P_{\Cal M}^m$ satisfying (1) and (2). Observe that there is an $\Cal L$-formula $\psi(x,y)$ such that an element $h \in M^m$ satisfies (2) if and only if $\Cal M\models \psi(h,b)$. By saturation, in order to find $h$ satisfying (1) and (2), it is enough to find for every $\Cal L'$-formula $\varphi(x)$ with $\Cal M\models \varphi_e(g)$ an $h\in P_{\Cal M}^m$ such that $\Cal M \models \varphi_e(h)\wedge \psi(h,b)$. So let $\varphi(x)$ be an $\Cal L'$-formula with $\Cal M\models \varphi_e(g)$. Consider the special formula $\chi(y)$ given by
\[
\exists x \ Px \wedge \varphi_e(x) \wedge \psi(x,y).
\]
Since $\Cal M\models \chi(a)$ and $a$ and $b$ satisfy the same special formulas, we get that $\Cal M\models \chi(b)$. Thus there exists $h\in P_{\Cal M}^m$ such that $\Cal M \models \varphi_e(h)\wedge \psi(h,b)$.
\end{proof}

\subsection{Types} In order to show statements (2)-(4) of Theorem A we need better control over the $\Cal L^*$-types in models of $T^*$. We establish the necessary results in this section. Throughout let $\Cal M$ be a $\kappa$-saturated model of $T^*$. We first introduce the following notation: for $C \subseteq \Cal M$ and $n\in \N$ we denote by $D_n(C)$ the set
\[
\{ z \in \Cal M^n \ : \ z \hbox{ is $\dcl_T$-independent from } CP_{\Cal M}\}.
\]

\begin{prop}\label{prop:type} Let $a \in P_{\Cal M}^k$, $z \in D_l(\emptyset)$, $b \in P_{\Cal M}^m$ and $y\in D_n(z)$. Then $\tp_{\Cal L^*}(by|az)$ is implied by the conjunction of
\begin{itemize}
\item $\tp_{\Cal L}(by|az)$,
\item ``$b\in P_{\Cal M}^m$'' and $\tp_{\Cal L'}(b|a)$,
\item ``$y \in D_n(z)$''.
\end{itemize}
\end{prop}
\begin{proof}
Set $\Cal M_1 := \Cal M_2 :=\Cal M$ and let $\Cal I$ be the back-and-forth system between $\Cal M_1$ and $\Cal M_2$ constructed in the previous section.
Let $b_1,b_2 \in P_{\mathcal M}^m$ and $y_1,y_2 \in D_n(z)$ be such that  $\tp_{\Cal L}(b_1y_1|az) = \tp_{\Cal L}(b_2y_2|az)$ and
 $\tp_{\Cal L'}(b_1 | a)= \tp_{\Cal L'}(b_2 | a)$. In order to show that $\tp_{\Cal L^*}(b_1y_1 | az) = \tp_{\Cal L^*}(b_2y_2 | az)$, we only need to find $\iota\in \Cal I$ such that $\iota(b_1y_1)=b_2y_2$ and the coordinates of $a$ and $z$ are in the domain of $\iota$. It is immediate that the identity on $\dcl_T(az)$ is in $\Cal I$. Since $\tp_{\Cal L}(b_1y_1|az)=\tp_{\Cal L}(b_2y_2|az)$, there is a partial $\Cal L$-isomorphism from $\dcl_T(azb_1y_1)$ to $\dcl_T(azb_2y_2)$ mapping $b_1y_1$ to $b_2y_2$. Because $\tp_{\Cal L'}(b_1 | a)= \tp_{\Cal L'}(b_2 | a)$ and $y_1,y_2 \in D_n(z)$, it is immediate that $\iota \in \Cal I$.
\end{proof}

\noindent We immediately obtain the following three corollaries from Proposition \ref{prop:type}.

\begin{cor}\label{lem:typeI} Let $C\subseteq \Cal M$ be finite and $y \in D_n(C)$. Then $\tp_{\Cal L^*}(y | C)$ is implied by $\tp_{\Cal L}(y | C)$ in conjunction with ``$y \in D_n(C)$''.
\end{cor}
\begin{cor}\label{lem:induced} Let $a \in P_{\Cal M}^k$, $z \in D_l(\emptyset)$ and $b \in P_{\Cal M}^n$. Then $\tp_{\Cal L^*}(b|az)$ is implied by $\tp_{\Cal L}(b|az)$, ``$b \in P_{\Cal M}^n$'' and $\tp_{\Cal L'}(b|a)$.
\end{cor}

\begin{cor} Let $Z \subseteq P_{\Cal M}^n$ be $\Cal L^*$-definable. Then there is an $\Cal L$-definable set $Y \subseteq \Cal M^n$ and an $\Cal L'$-formula $\varphi(x)$ such that
\[
Z = Y \cap \{ a \in P_{\Cal M}^n \ : \ \Cal M \models \varphi_e(a)\}.
\]
\end{cor}

\noindent Combining Proposition \ref{prop:type} with a result of Boxall and Hieronymi \cite{BoxallH}, we are now able to deduce statement (2) of Theorem A.

\begin{thm}\label{thm:opencoreTstar} The theory $T$ is an open core of $T^*$.
\end{thm}
\begin{proof}
We will use \cite[Corollary 3.1]{BoxallH} to show that every $\Cal L^*$-definable open set in $\Cal M$ is also $\Cal L$-definable. Let $X\subseteq \Cal M^n$ be open and $\Cal L^*$-definable over some finite parameter set $C$. We will now apply \cite[Corollary 3.1]{BoxallH}, using $D_n(C)$ as $D_{S_1\dots S_n}$. Therefore it is left to check that conditions (1)-(3) of \cite[Corollary 3.1]{BoxallH} hold for $D_n(C)$. These three conditions are
\begin{enumerate}
\item $D_n(C)$ is dense in $\Cal M$,
\item for every $y\in D_n(C)$ and every open set $U\subseteq \Cal M^n$, if $\tp_{\Cal L}(y|C)$ is realized in $U$, then $\tp_{\Cal L}(y|C)$ is realized in $U\cap D_n(C)$,
\item for every $y \in D_n(z)$, $\tp_{\Cal L^*}(y|C)$ is implied by $\tp_{\Cal L}(y | C)$ in conjunction with ``$y \in D_n(C)$''.
\end{enumerate}
Condition (1) follows easily from saturation of $\Cal M$ and \cite[2.1]{DMS-Indepedent}. Using o-minimality of $T$, it is easy to deduce Condition (2) from Condition (1). Finally, Condition (3) holds by Corollary \ref{lem:typeI}.
\end{proof}

\subsection{Completions of $T_e$} Using results from the previous sections we will now give a characterizations of all complete $\Cal L_e$-theories containing $T_e$.

\begin{defn}
Let $T_{e,\infty}$ be the $\Cal L_e$-theory consisting of $T_e$ and an axiom schema expressing the following statement:
\begin{itemize}
\item[(T6)] $E$ has infinitely many equivalence classes.
\end{itemize}
Similarly, for every $n\in \N_{>0}$ define $T_{e,n}$ to be the $\Cal L_e$-theory consisting of $T_e$ and a sentence stating that $E$ has exactly $n$ equivalence classes.
\end{defn}

\begin{thm}\label{thm:tecompletion} Let $p \in \N_{>0}\cup \{\infty\}$. The theory $T_{e,p}$ is complete.
\end{thm}
\begin{proof}
We first consider the case that $p=\infty$. Let $\Cal L'$ be empty and $T'$ be the (complete) $\Cal L'$-theory of infinite sets. Let $\Cal L^*$ and $\Cal T^*$ be constructed as above. Since $\Cal L'=\emptyset$, we have that $\Cal L^* = \Cal L_e$. Since $T^*$ is complete, it is enough to show that every model of $T_{e,\infty}$ is also a model of $T^*$. Let $\Cal M \models T_{e,\infty}$. Since $\Cal L'=\emptyset$, we immediately get that $\Cal M$ satisfies (T4). It is left to show that $\Cal M$ satisfies (T5). Let $\varphi \in T'$. Since $T'$ is the theory of infinite sets, there is $n\in \N$ such that $\varphi$ is the following formula
\[
\exists x_1 \dots \exists x_n \bigwedge_{1\leq i<j \leq n} x_i \neq x_j.
\]
It is easy to check that $\varphi_e$ is the $\Cal L_e$-formula
\[
\exists x_1 \dots \exists x_n \bigwedge_{1\leq i<j \leq n} Px_i \wedge \neg Ex_ix_j.
\]
Since $\Cal M$ satisfies (T6), we get that $\Cal M \models \varphi_e$. Thus $\Cal M$ satisfies (T5).\newline
\noindent The proof of the case $p\in \N_{>0}$ can be done similarly by replacing the $\Cal L'$-theory of infinite sets by the $\Cal L'$-theory of a set with exactly $p$ elements.
\end{proof}

\noindent From Theorem \ref{thm:tecompletion} we can directly deduce the following characterization of completions of $T_e$.

\begin{cor} Let $\tilde{T}$ be a complete $\Cal L_e$-theory such that $T_e \subseteq \tilde{T}$. Then there is $p \in \N_{>0} \cup \{\infty\}$ such that
$T_{e,p}\models \tilde{T}$.
\end{cor}

\noindent We obtain the following corollary as an immediate consequence of Theorem \ref{thm:opencoreTstar} and the proof of Theorem \ref{thm:tecompletion}.

\begin{cor} Let $p \in \N_{>0}\cup \{\infty\}$. The theory $T$ is an open core of $T_{e,p}$.
\end{cor}


\section{Preservation of NIP}\label{Section:NIP}

Let $T$ be a complete o-minimal extension of the theory of densely ordered
groups with a distinguished positive element, and let $\Cal L$ be its language. As before, let $\Cal L'$ be a relational language disjoint from $\Cal L$, and let $T'$ be a complete $\Cal L'$-theory. Furthermore, let $T^*$ be the $\Cal L^*$-theory constructed in the previous section. We will now show that $T^*$ is NIP if $T'$ is NIP. As we will see, this can be deduced rather directly from Corollaries \ref{lem:typeI} and \ref{lem:induced} and the following result of G{\"u}nayd\i n and Hieronymi \cite{GH-Dependent}.

\begin{fact}\label{gh:fact}\cite[Proposition 2.4]{GH-Dependent} Let $\Cal L_0$ be a first-order language and let $\Cal L_1$ be a language containing $\Cal L_0$ and a unary predicate symbol $U$ not in $\Cal L_0$. Let $T_0$ be a complete $\Cal L_0$-theory and let $T_1$ be a complete $\Cal L_1$-theory extending $T_0$. Let $\mathbb M$ be a monster model of $T_1$. Suppose that
\begin{itemize}
\item[(i)] $\dcl_{T_0}$ is a pregeometry,
\item[(ii)] for every $\Cal L_1$-formula $\varphi(x,y)$, indiscernible sequence $(g_i)_{i\in\omega}$ from $U_{\mathbb M}^p$ and $b\in\mathbb M^q$, the set
$\{ i\in\omega \ : \ \mathbb{M} \models
  \varphi(g_i,b)\}$
is either finite or co-finite (in $\omega$),
\item[(iii)] for every formula $\varphi(x,y)$, indiscernible sequence $(a_i)_{i\in\omega}$ from $\mathbb M$ and $b\in \mathbb M^q$ with $a_i\notin\dcl_{\Cal L_0}(U_{\mathbb M}b)$ for every $i\in\omega$, the set
$\{ i\in\omega \ : \ \mathbb{M} \models
  \varphi({a_i},b)\}$
is either finite or co-finite (in $\omega$).
\end{itemize}
Then $T_1$ is NIP.
\end{fact}

\begin{thm} If $T'$ is NIP, so is $T^*$.
\end{thm}
\begin{proof}
We apply Fact \ref{gh:fact} with $T_0:=T$ and $T_1:= T^*$. Since $T$ is o-minimal, $\dcl_T$ is a pregeometry.\newline
For (ii), let $\varphi(x,y)$ be an $\Cal L^*$-formula, $(g_i)_{i\in\omega}$ an indiscernible sequence from $P_{\mathbb M}^p$ and $b\in\mathbb M^q$. Without loss of generality, we can assume that there are $b_1\in P_{\mathbb M}^{q_1}$ and $b_2 \in \mathbb M^{q_2}$ such that $b_2$ is $\dcl_T$-independent over $P_{\mathbb M}$ and $b=(b_1,b_2)$.
By Corollary \ref{lem:induced} there is an $\Cal L$-formula $\psi(x,u,v)$ and an $\Cal L'$-formula $\theta(x,u)$ such that for all $a \in P_{\mathbb M}^p$
\begin{equation}\label{eq:NIP}
\tag{$\ast$} \mathbb{M} \models \varphi(a,b) \leftrightarrow \big(\psi(a,b_1,b_2) \wedge \theta_e(a,b_1)\big).
\end{equation}
Since both $T$ and $T'$ are NIP, it follows immediately from \eqref{eq:NIP} that $\{ i\in\omega \ : \ \mathbb{M} \models \varphi(g_i,b)\}$ is either finite or co-finite.\newline
For (iii), let $\varphi(x,y)$ be an $\Cal L^*$-formula, $(a_i)_{i\in\omega}$ an indiscernible sequence from $\mathbb M$ and $b\in \mathbb M^q$ with $a_i\notin\dcl_{T}(P_{\mathbb M}b)$ for every $i\in\omega$.
By Corollary \ref{lem:typeI} there is an $\Cal L$-formula $\psi(x,b)$ such that for all $a \in \mathbb{M} \setminus \dcl_{T}(P_{\mathbb M}b)$
\[
\mathbb{M} \models \varphi(a,b) \leftrightarrow \psi(a,b).
\]
Since $T$ is NIP, $\{ i\in\omega \ : \ \mathbb{M} \models   \varphi({a_i},b)\}$ is either finite or co-finite (in $\omega$).
\end{proof}

\noindent We can now give a proof of Theorem D that there is an NIP expansion of $T$ without a distal expansion.

\begin{proof}[Proof of Theorem D]
Fix a prime $p$. Let $T'$ be $ACF_p$. Since $T'$ is stable, $T'$ is NIP. Suppose $T^*$ has a distal expansion $\tilde{T}$. Then $\tilde{T}^{eq}$ is distal by \cite[Remark after Definition 9.17]{Simon-Book}. However, by Proposition \ref{prop:interpret} every model of $\tilde{T}^{eq}$ defines a model of $T'$. By \cite[Proposition 6.2]{CS-RegDis} $\tilde{T}^{eq}$ cannot be distal. A contradiction.
\end{proof}

\section{Preservation of Strong Dependence}\label{Section:STDEP}

In this section, we will show that $T^*$ (as constructed in Section \ref{Section:Fusion}) is strongly dependent if $T'$ is. We essentially follow the proof of Berenstein, Dolich and Onshuus \cite[Theorem 2.11]{BDO}.
\newline

\noindent Let $\Cal L_0$ be a first-order language containing $<$ and let $\Cal L_1$ be a language containing $\Cal L_0$ and a unary predicate symbol $U$ not in $\Cal L_0$. Let $T_0$ be a complete $\Cal L_0$-theory extending the theory of linear ordered sets such that $\dcl_{T_0}$ is a pregeometry. Let $T_1$ be a complete $\Cal L_1$-theory extending $T_0$, and let $\mathbb M$ be a monster model of $T_1$. If $X,Y$ are subsets of $\mathbb M$, we say $X$ is \textbf{$U$-independent} over $Y$ if $X\setminus U_{\mathbb M}$ is $\dcl_{T_0}$-independent over $U_{\mathbb M}Y$. If $Y=\emptyset$, we simply say that $X$ is $U$-independent. We say an indiscernible sequence  $(a_i)_{i \in I}$ of tuples of elements of $\mathbb M$ is \textbf{$U$-independent} if each $a_i$ is $U$-independent.

\begin{lem}\label{lem:indi} Let $\kappa$ be an infinite cardinal and let $(a_i)_{i \in I} $ be an indiscernible sequence of tuples of elements of $\mathbb M$ of length $\kappa$. Then there is an $U$-independent indiscernible sequence $J=(b_i)_{i \in I}$ of tuples of elements of $\mathbb M$ of length $\kappa$ such that for every $j<\kappa$ there is an $\Cal L_0$-$\emptyset$-definable function $f : \mathbb M^n \to \mathbb M$, and $j_1,\ldots, j_n<\kappa$ such that for every $i \in I$
\[
a_{i,j} = f(b_{i,j_1},\dots,b_{i,j_n}).
\]
\end{lem}

\begin{proof}
We inductively construct a sequence $(b_i)_{i\in I}$ from the sequence $(a_i)_{i \in I}$ by removing $U$-dependencies. Let $\alpha< \kappa$ be minimal such that there is $i\in I$ such that $\{a_{i,j}:j\leq \alpha\}$ is not $U$-independent. By minimality of $\alpha$ there are $j_1<\ldots < j_m < \alpha$ and an $\Cal L_0$-$\emptyset$-definable $f:M^{m+\ell+1}\to M$ such that
\begin{equation}\label{eq:indi}\tag{$*$}
\exists u_{i,0},\ldots, u_{i,\ell}\in U_{\mathbb M} \ f(a_{i,j_1},\ldots,a_{i,j_m},u_{i,0},\ldots,u_{i,\ell}) = a_{i,\alpha}.
\end{equation}
By indiscernibility of $(a_i)_{i\in I}$, \eqref{eq:indi} holds for every $i\in I$. For each $i\in I$, define a set
\[
S_i = \{ (u_{0},\ldots,u_{\ell}) \in U^{l+1} \ : \ f(a_{i,j_1},\ldots,a_{i,j_m},u_{0},\ldots,u_{\ell}) = a_{i,\alpha}\}.
\]
By indiscernibility of $(a_i)_{i\in I}$, we have that $S_i$ is finite for some $i$ if and and only $S_{i}$ is finite for every $i\in I$. \newline
We first consider the case that $S_i$ is finite for every $i \in I$. Then for each $i\in I$ we may choose $u_i = (u_{i,0},\ldots u_{i,\ell})$ to be the lexicographically least member of $S_i$. Let $b_i$ be the tuple where $a_{i,\alpha}$ is replaced by $u_i$. As $a_{i,\alpha}$ and $u_i$ are interdefinable over $\{a_{i,j_1},\ldots, a_{i,j_m}\}$, $(b_i)_{i\in I}$ is indiscernible. Furthermore, the set $\{b_{i,1},\dots,b_{i,\alpha+\ell}\}$ is $U$-independent for each $i\in I$.\newline
Now suppose that $S_i$ is infinite. Consider the collection of formulas in variables $(x_{i,j})_{i \in I}$ for $j<\kappa$ stating:

\begin{enumerate}

\item $x_{i,j}=a_{i,j}$ for $j<\alpha$

\item $f(a_{i,j_1},\ldots a_{i,j_m},x_{i,\alpha},\ldots, x_{i,\alpha+\ell})=a_{i,\alpha}$

\item $x_{i,\alpha},\ldots, x_{i,\alpha+\ell} \in U$

\item $x_{i,\alpha+\ell+j} = a_{i,\alpha+j}$ for $j \ge 1$

\item The sequence $(x_i)_{i \in I}$ is indiscernible.

\end{enumerate}
As $S_i$ is infinite, it can be shown by a standard argument using Ramsey's theorem that this collection is finitely satisfiable. Therefore, by saturation there is a realization $(b_i)_{i \in I}$ of this collection. By construction, we have for every $i\in I$ that $\{b_{i,1},\dots, b_{i,\alpha + \ell}\}$ is $U$-independent and $a_{i,\alpha} = f(b_{i,j_1},\ldots b_{i,j_m},b_{i,\alpha},\ldots, b_{i,\alpha+\ell})$. \newline
Inductively continuing, we arrange the sequence $(b_i)_{i\in I}$ as desired.
\end{proof}

\noindent We will use Lemma \ref{lem:indi} to show a criterion for strong dependence for $T_1$. Before we do so, we recall the definition of strong dependence. If $I$ is a linear order, we denote its completion by $\comp(I)$. If $I$ is a linear order, $c=(c_1,\dots,c_n)\in \comp(I)^n$ and $i,i' \in I$, we write $i \sim_c i'$  if
\[
\bigwedge_{j=1}^n \big((i < c_j \leftrightarrow i' < c_j) \wedge (i = c_j \leftrightarrow i' = c_j)\big)
\]
Note that $\sim_{c}$ defines an equivalence relation $\sim_{c}$ on $I$.

\begin{defn} A theory $\tilde{T}$ in a language $\tilde{\Cal L}$ is \textbf{strongly dependent} if for every $\mathcal{M}\models \tilde{T}$, every $b\in \mathcal M^m$ and every indiscernible sequence $(a_i)_{i \in I}$, there is $n\in \N$ and $c\in  \comp(I)^n$ such that $i\sim_c j \Rightarrow \tp_{\tilde{\Cal L}}(a_i | b) = \tp_{\tilde{\Cal L}}(a_j|b)$.
\end{defn}

\noindent For more details and other equivalent definitions of strong dependence, we refer the reader to \cite[Chapter 4]{Simon-Book}.

\begin{lem}\label{lem:stdep} The following are equivalent:
\begin{itemize}
\item[(i)] For every $b\in \mathbb M$, and every $U$-independent indiscernible sequence $(a_i)_{i \in I}$, if $b$ is $U$-independent over $\{a_i : i \in I\}$, then there is $n\in \N$ and $c\in  \comp(I)^n$ such that $i\sim_c j \Rightarrow \tp_{\Cal L_1}(a_i | b) = \tp_{\Cal L_1}(a_j|b)$.
\item[(ii)] For every $b\in \mathbb M$, and indiscernible sequence $(a_i)_{i \in I}$, if $b$ is $U$-independent over $\{a_i : i \in I\})$, then there is $n\in \N$ and $c\in  \comp(I)^n$ such that $i\sim_c j \Rightarrow \tp_{\Cal L_1}(a_i | b) = \tp_{\Cal L_1}(a_j|b)$.
\item[(iii)] $T_1$ is strongly dependent.
\end{itemize}
\end{lem}
\begin{proof}
It is clear that (iii)$\Rightarrow$(ii)$\Rightarrow$(i). Observe that (i)$\Rightarrow $(ii) follows easily from Lemma \ref{lem:indi}. So we only need to show that (ii) implies (iii). Let $b \in \mathbb M^m$ and $(a_i)_{i \in I}$ be an indiscernible sequence of possibly infinite tuples from $\mathbb M$. It is enough to consider the case $m=1$ (see for example \cite[Proposition 4.26]{Simon-Book}). Suppose $b\in \dcl_{T_0}(U_{\mathbb M}\{a_i : i \in I\})$. Then there are $g \in U_{\mathbb M}^l$, $e=(e_1,\dots, e_k)\in I^k$ and an $\Cal L_0$-$\emptyset$-definable function $f$ such that
\begin{equation}\label{eq:stdep}\tag{$\dagger$}
b=f(g,a_{e_1},\dots, a_{e_k}).
\end{equation}
Without loss of generality assume that $e_1 < \dots < e_k$. Set $a_e=(a_{e_1},\dots, a_{e_k})$, $e_0=-\infty$ and $e_{k+1}=+\infty$. Let $t \in \{0,\dots,k\}$. Now observe that $(a_ea_j)_{j \in (e_{t},e_{t+1})\cap I}$ is an indiscernible sequence. By (ii) there is $d_t \in (\comp(I)\cap (e_{t},e_{t+1}))^{n_t}$ such that for all $i,j \in (e_{t},e_{t+1})$ we have $i \sim_{d_t} j \Rightarrow \tp_{\Cal L_1}(a_{e}a_i|g) = \tp_{\Cal L_1}(a_{e}a_j|g)$. By \eqref{eq:stdep} we get that for all such $i,j$
\[
i \sim_{d_t} j \Rightarrow \tp_{\Cal L_1}(a_i|b) = \tp_{\Cal L_1}(a_j|b).
\]
Set $c:=(d_0e_1d_1\dots e_kd_{k+1})$. It can be checked easily that this is the desired $c\in \comp(I)^n$.
\end{proof}

\noindent Let us now recall the setting of Section \ref{Section:Fusion}. Let $T$ be a complete o-minimal extension of the theory of densely ordered
groups with a distinguished positive element, and let $\Cal L$ be its language. As before, let $\Cal L'$ be a language disjoint from $\Cal L$, and let $T'$ be a complete $\Cal L'$-theory. Furthermore, let $T^*$ be the $\Cal L^*$-theory constructed in Section \ref{Section:Fusion}.

\begin{thm}
If $T'$ is strongly dependent, so is $T^*$.
\end{thm}
\begin{proof}
We now apply Lemma \ref{lem:stdep} with $T_0:=T$, $T_1:=T^*$ and $U:=P$. As before, note that $\dcl_T$ is a pregeometry, since $T$ is o-minimal.\newline
\noindent Let $b\in \mathbb M$ and $(a_i)_{i \in I}$ be an $P$-independent sequence such that $b$ is $P$-independent over $\{a_i : i \in I\}$. Since each $a_i$ is $P$-independent, we have (after possibly changing the order of entries of the $a_i$'s) that for each $i\in I$ there are tuples $u_i,v_i$ of elements of $\mathbb M$ such that for each $i \in i$
\begin{itemize}
\item $a_i = u_iv_i$,
\item $u_i$ is a tuple of elements in $P_{\mathbb M}$,
\item $v_i$ is $\dcl_T$-independent over $P_{\mathbb M}$.
\end{itemize}
Since $b$ is $P$-independent over $\{a_i : i \in I\}$, we get that either $b\in P_{\mathbb M}$ or $b\notin \dcl_T(\{a_i  :  i \in I\}P_{\mathbb{M}})$. We consider the two different cases.\newline

\noindent Let $b\in P_{\mathbb M}$. By Proposition \ref{prop:type} the type $\tp_{\Cal L^*}(u_iv_i|b)$ is determined by
\begin{itemize}
\item $\tp_{\Cal L}(u_iv_i|b)$,
\item the statement ``$u_i$ is a tuple of elements of $P_{\mathbb{M}}$'' and $\tp_{\Cal L'}(u_i|b)$,
\item the statement ``$v_i$ is $\dcl_T$-independent over $P_{\mathbb M}$''.
\end{itemize}
Since both $T$ and $T'$ are strongly dependent, we can find $c\in \comp(I)^n$ such that for every $i,j \in I$
\[
i\sim_c j \Rightarrow \Big( \tp_{\Cal L}(u_iv_i|b)=\tp_{\Cal L}(u_jv_j|b) \hbox{ and } \tp_{\Cal L'}(u_i|b)=\tp_{\Cal L'}(u_j|b) \Big).
\]
Thus for every $i,j \in I$ with $i\sim_c j$ we get $\tp_{\Cal L^*}(a_i|b)=\tp_{\Cal L^*}(a_j|b)$.\newline

\noindent Now suppose that $b\notin \dcl_T(\{a_i  :  i \in I\}P_{\mathbb{M}})$. In particular, $b \notin \dcl_T(P_{\mathbb{M}})$. Since $\dcl_T$ is a pregeometry, $v_i$ is $\dcl_T$-independent over $P_{\mathbb{M}}b$ for each $i\in I$. By
Proposition \ref{prop:type}, for each $i\in I$ the type $\tp_{\Cal L^*}(u_iv_i|b)$ is determined by
\begin{itemize}
\item $\tp_{\Cal L}(u_iv_i|b)$,
\item the statement ``$u_i$ is a tuple of elements of $P_{\mathbb{M}}$'' and $\tp_{\Cal L'}(u_i)$,
\item the statement ``$v_i$ is $\dcl_T$-independent over $P_{\mathbb M}b$''.
\end{itemize}
As before using strong dependence of $T$ and $T'$, we can find $c\in \comp(I)^n$ such that for every $i,j \in I$
\[
i\sim_c j \Rightarrow \Big( \tp_{\Cal L}(u_iv_i|b)=\tp_{\Cal L}(u_jv_j|b) \hbox{ and } \tp_{\Cal L'}(u_i)=\tp_{\Cal L'}(u_j) \Big).
\]
Thus for every $i,j \in I$ with $i\sim_c j$ we get $\tp_{\Cal L^*}(a_i|b)=\tp_{\Cal L^*}(a_j|b)$.
\end{proof}

\noindent This completes the proof of Theorem A. In the next section we will deduce Theorem B from Theorem A. \newline

\noindent It is worth pointing out in this section on strong dependence that by Dolich and Goodrick \cite[Corollary 2.4]{DG} every strongly dependent expansion of the real field has o-minimal open core. In contrast to this restriction, our Theorem B(4) shows that there is a large variety of such expansions of the real field.

\section{Proof of Theorem B}

The purpose of this section is twofold. We first deduce Theorem B from our proof of Theorem A. Then we show that in Theorem B the statement ``$\Cal S$ interprets a model of $T'$'' cannot be replaced by the statement ``$\Cal S$ defines a model of $T'$''.

\begin{proof}[Proof of Theorem B]
Let $\Cal R=(\R,<,+,\dots)$ be an o-minimal expansion of the real ordered additive group in a language $\Cal L$ and let $T'$ be a theory such that $|\Cal L|<|\R|$ and $|T'|\leq |\R|$. Let $T^*$ be the theory as constructed in Section \ref{Section:Fusion}. Since $T^*$ satisfies the statements (1)-(4) of Theorem A, it is only left to show that $\Cal R$ can be expanded to a model of $T^*$.
Since $|\Cal L|<|\R|$, we can find a $\dcl_T$-basis of cardinality at least $|T'|$. Since $\dcl_T(\emptyset)$ is dense in $\R$, we are able to choose this basis such that it is dense in $\R$. Now apply Lemma \ref{lem:expand}.
\end{proof}

\begin{prop}\label{prop:ocdlo} Let $\Cal S$ be an expansion of $(\R,<,+)$.
The following are equivalent
\begin{enumerate}
\item $\Cal S$ defines an infinite discrete linear order.
\item $\Cal S$ defines an order with order type $\omega$.
\item The open core of $\Cal S$ is not o-minimal.
\end{enumerate}
\end{prop}

\begin{proof}
We show that (1) implies (2).
Suppose $\Cal S$ defines an infinite discrete linear order $(D,\prec)$.
Fix $d \in D$.
Either $D_{ \prec d}$ or $D_{\succ d}$ is infinite.
After replacing $\prec$ with the reverse order if necessary, we may suppose that $D_{ \succ d}$ is infinite.
After replacing $(D,\prec)$ with $(D_{\succeq d }, \prec)$ if necessary we suppose that $(D, \prec)$ has a minimal element.
Let $E \subseteq D$ be the set of $e$ such that $D_{\prec e}$ is finite.
Recall that a subset of $\R^n$ is finite if and only if it is closed, bounded and discrete.
It follows that $E$ is definable.
Note that $E_{\prec e}$ is finite for all $e \in D$.
Then $(E, \prec)$ is a discrete linear order with minimal element and finite initial segments. Thus it has order type $\omega$.\newline
We now show that (2) implies (3).
Suppose that $(D, \prec)$ is a definable order with order type $\omega$ and $D \subseteq \R^n$.
First suppose that there is no coordinate projection $\pi: \R^n \to \R$ such that $\pi(D)$ is somewhere dense. Since $D$ is infinite, there is a coordinate projection $\rho : \R^n \to \R$ such that $\rho(D)$ is infinite. Then $\rho(D)$ is an infinite, nowhere dense, subset of $\R$. Thus the open core of $\Cal S$ is not o-minimal. \newline
Now let $\pi: \R^n \to \R$ be a coordinate projection such that $\pi(D)$ is somewhere dense. Let $a,b\in \R$ such that $(a,b)$ is an interval in in which $\pi(D)$ is dense.
We now reduce to the case when $D$ is a dense subset of an open interval.
Note that $D' = \{ e \in D : a < \pi(e) < b \}$ is an infinite, and hence $\prec$-cofinal, subset of $D$.
It follows that $(D',\prec)$ has order type $\omega$.
After replacing $D$ with $D'$ if necessary we suppose that $\pi(D)$ is a subset of $(a,b)$.
We put an order $\prec_\pi$ on $\pi(D)$ by declaring $x \prec_\pi y$ if there is a $e \in D$
such that $\pi(e) = x$ and $\pi(e') \neq y$ for all $e' \prec e$.
It is easy to see that $(\pi(D), \prec_\pi)$ has order type $\omega$.
After replacing $(D,\prec)$ with $(\pi(D), \prec_\pi)$ we suppose that $D$ is a dense subset of $(a,b)$.
We declare
\[
Y := \{ x \in D \ : \ \forall z \in D (z\prec x) \rightarrow (z < x) \}.
\]
That is, $Y$ is the set of $e \in D$ such that $e$ the $<$-maximal element of $D_{\preceq e}$.
By density of $D$ in $(a,b)$, it is easy to see that $Y$ is infinite and that $(Y, <)$ is order-isomorphic to $(\N, <)$.
Thus $Y$ is an infinite discrete definable subset of $\R$. Hence the closure of $Y$ does not have interior, but infinitely many connected components. Therefore $\Cal S$ does not have o-minimal open core.\newline

\noindent  Since (2) trivially implies (1), it is enough to show that (3) implies (2). Suppose that the open core of $\Cal S$ is not o-minimal. By \cite[2.14 (2)]{DMS1} there is an infinite discrete subset $D\subseteq \R$ definable in $\Cal S$. First consider the case that $D\cap [-a,a]$ is a finite set for every $a \in \R_{>0}$. Then either $((-D) \cap [0,\infty),<)$ or $(D\cap[0,\infty),<)$ has order type $\omega$. From now on we can assume that there is $a\in \R_{>0}$ such that the cardinality of $D\cap [-a,a]$ is infinite. Thus without loss of generality we can assume that $D$ is bounded. For $\varepsilon \in \R_{>0}$ set
\[
D_{\varepsilon} := \{ d \in D \ : \ (d-\varepsilon,d+\varepsilon) \cap D = \{d\}\}.
\]
Since $D$ is bounded, each $D_{\varepsilon}$ is finite. Moreover, since $D$ is discrete and infinite, there is a function $f: D\to \R_{>0}$ definable in $\Cal S$ mapping $d\in D$ to the supremum of all $\varepsilon \in \R_{>0}$ with $d\in D_{\varepsilon}$. We now define the following order on $D$: let $d_1,d_2\in D$, we set $d_1 \prec d_2$ whenever
one of the following conditions holds:
\begin{itemize}
\item $f(d_1) > f(d_2)$,
\item $f(d_1)=f(d_2)$ and $d_1 < d_2$.
\end{itemize}
It can be checked easily that $(D,\prec)$ has order type $\omega$.
\end{proof}

\noindent Let $T'$ be the theory of an infinite discrete order. By Theorem B there exists an expansion of $(\R,<,+)$ that has o-minimal open core and interprets a model of $T'$. However, by Proposition \ref{prop:ocdlo} there is no expansion of $(\R,<,+)$ that has o-minimal open core and defines a model of $T'$. Therefore in Theorem B the statement ``$\Cal S$ interprets a model of $T'$'' cannot be replace by the statement ``$\Cal S$ defines a model of $T'$''.

\section{Noiseless NIP expansions of $(\R,<,+)$}

Recall that an expansion of $(\R,<)$ is noiseless if it does not define a somewhere dense and co-dense subset of $\R$. In this section we show that every noiseless NIP expansion of $(\R,<,+,1)$ has definable choice and hence eliminates imaginaries. This statement will be established for the slightly larger class of noiseless expansions of $(\R,<,+,1)$ that do not define a Cantor set. A \textbf{Cantor set} is a non-empty compact subset of $\R$ that neither has interior nor isolated points. By \cite[Theorem B]{HW-Monadic} every NTP$_2$ (and hence every NIP) expansion of $(\R,<,+)$ does not define a Cantor set.\newline

\noindent Fix a noiseless expansion $\Cal R$ of $(\R,<,+,1)$ that does not define a Cantor set. Throughout this section, \emph{definable} will mean \emph{definable in $\Cal R$}. For a subset $X\subseteq \R^n$, we denote the (topological) closure of $X$ by $\cl(X)$ and the interior of $X$ by $\Int(X)$.

\begin{lem}\label{lem:isolatedpoint} Let $X\subseteq \R$ be a non-empty definable set with empty interior. Then $X$ contains an isolated point.
\end{lem}
\begin{proof}
Since $\Cal R$ is noiseless, the closure $\cl(X)$ of $X$ has empty interior. Because $\Cal R$ does not define a Cantor set, $\cl(X)$ has an isolated point. It follows directly that $X$ has an isolated point.
\end{proof}

\noindent Therefore in an expansion of $(\R,<,+)$ that does not define a Cantor set, every definable subset of $\R$ contains a locally closed point. For expansions of the real field, the existence of definable Skolem functions in expansions satisfying the latter condition was shown in \cite[Lemma 9.1]{F-tameopen}.

\begin{lem}\label{lem:nowheredense} Let $C\subseteq \R^{n+1}$ be $\emptyset$-definable such that $C_x$ has empty interior for every $x\in \R^n$. Then there is an $\emptyset$-definable function $f: \pi(C) \to \R$ such that $\gr(f) \subseteq C$, where $\pi : \R^{n+1} \to \R^n$ is the projection onto the first $n$ coordinates.
\end{lem}
\begin{proof}
By Lemma \ref{lem:isolatedpoint} we have that for all $x\in \pi(C)$ the set $C_x$ has an isolated point whenever $C_x$ is non-empty. Let $g : \pi(C) \to \R$ map $x\in \pi(C)$ to
\[
\sup \{r \in \R_{>0} \ : \ \exists y \in C_x \ (y-r,y+r) \cap C_x = \{y\} \}
\]
if such supremum exists, and to $1$ otherwise. Define
\[
D:= \{ (x,y) \in C \ : \  (y-\tfrac{g(x)}{2},y+\tfrac{g(x)}{2}) \cap C_x = \{y\}\}.
\]
It is easy to check that $D_x$ is non-empty if and if $C_x$ is non-empty.
For each $x\in \pi(D)$ and $y_1,y_2 \in D_x$, we have $|y_1-y_2| \geq \tfrac{g(x)}{2}$. Therefore
the set $D_x$ is closed and discrete for each $x \in \pi(D)$. Let $f: \pi(C) \to \R$ be the function defined by
\[
x\mapsto \left\{
\begin{array}{ll}
    \min D_x \cap [0,\infty), &  \hbox{if $D_x \cap [0,\infty)$ is non-empty} \\
     \max D_x \cap (-\infty,0), &  \hbox{otherwise.}
    \end{array}
  \right.
\]
Observe that $f$ is well-defined, because $D_x$ is closed and discrete. From the definition of $f$ we obtain directly that $\gr(f)\subseteq C$.
\end{proof}

\begin{prop}\label{prop:skolem} Let $A\subseteq \R^{m}\times \R^n$ be $\emptyset$-definable. Then there is an $\emptyset$-definable function $f: \pi(A) \to \R^n$ such that
\begin{enumerate}
\item $\gr(f)\subseteq A$,
\item $f(a) = f(b)$ whenever $a,b\in \pi(A)$ and $A_a=A_b$,
\end{enumerate}
where $\pi : \R^{m+n} \to \R^m$ is the projection onto the first $m$ coordinates.
\end{prop}
\begin{proof}
Using induction it is easy to reduce to the case that $n=1$. We can split $A$ into $B, C\subseteq \R^{m+n}$ such that $A=B\cup C$ and
\[
B := \{(x,y) \in A \ : \ y \in \Int(A_x)\}, \quad C := \{(x,y) \in A \ : \ y \in A_x \setminus \Int(A_x)\}.
\]
Observe that $C_x$ has empty interior for each $x\in \pi(C)$. Thus by Lemma \ref{lem:nowheredense} there is a definable function $f_1 : \pi(C) \to \R$ such that $\gr(f_1) \subseteq C$. Now define a subset $D\subseteq \R^{m+1}$ such that $(x,y)\in D$ whenever one the following conditions holds:
\begin{itemize}
\item  $y$ is a midpoint of a connected component of $B_x$,
\item $y = 1 + \sup (\R\setminus B_x)$ and $\R\setminus B_x$ is bounded from above,
\item $y=-1 + \inf (\R \setminus B_x)$ and $\R \setminus B_x$ is
bounded from below,
\item $y= 0$ and $B_x = \R$.
\end{itemize}
It is easy to see that $D$ is definable, $D\subseteq B$ and $\pi(B)=\pi(D)$. Moreover, $D_x$ has empty interior for each $x\in \pi(D)$. By Lemma \ref{lem:nowheredense} there is a definable function $f_2: \pi(D) \to \R$ such that $\gr(f_2) \subseteq D$. We now define $f: \pi(A) \to \R$ by
\[
x \mapsto \left\{
            \begin{array}{ll}
              f_1(x), & \hbox{if $B_x=\emptyset$;} \\
              f_2(x), & \hbox{otherwise.}
            \end{array}
          \right.
\]
It follows directly that $\gr(f)\subseteq A$. Furthermore, the reader can easily check that for $a\in \pi(A)$ the value of $f(a)$ only depends on $A_a$ and not on $a$. Therefore condition (2) holds for $f$ as well.
\end{proof}

\noindent Theorem E follows immediately from Proposition \ref{prop:skolem}. Note that Theorem E fails for NIP expansions of $(\R,<,+,1)$ in general. For example, the structure $(\R,<,+,1,\Q)$ is NIP (see for example \cite[Corollary 3.2]{GH-Dependent}), does not have definable Skolem functions (\cite[5.4]{DMS1}) and does not eliminate imaginaries (\cite[5.5]{DMS1}).\newline

\bibliographystyle{abbrv}
\bibliography{HW-Bib}

\end{document}